\documentclass[a4paper]{amsart}
\usepackage[utf8]{inputenc}
\usepackage[T1]{fontenc}
\usepackage{lmodern}
\usepackage{microtype}
\usepackage{mathrsfs}		
\usepackage{csquotes}		
\usepackage{mathtools}		
\usepackage{enumitem}		
\usepackage[ngerman,american]{babel}	
\usepackage[all]{xy}					
\usepackage[
backend=biber,
giveninits=true,
isbn=false,
url=false
]{biblatex}
\usepackage{hyperref}
\usepackage[
nameinlink,
]{cleveref}
\allowdisplaybreaks
\SelectTips{cm}{10}			
\theoremstyle{plain}
\newtheorem{theo}{Theorem}
\newtheorem{lemm}{Lemma}
\newcommand{\ZZ}{\mathbb{Z}}				
\newcommand{\CC}{\mathbb{C}}				

\newcommand{\id}[1]{\mathsf{id}_{#1}}		
\newcommand{\pr}[1]{\mathsf{pr}_{#1}}		
\newcommand{\PP}[1]{\mathbb{P}(#1)}			
\newcommand{\CP}[1]{\mathbb{P}^{#1}}		
\newcommand{\Tsp}[2]{\mathsf{T}_{#2}#1}		
\newcommand{\Hsh}[3]{\mathsf{H}^{#1}(#2;#3)}		
\newcommand{\hdim}[3]{\mathsf{h}^{#1}(#2;#3)}		
\newcommand{\Hdg}[3]{\mathsf{H}^{#1,#2}(#3)}		
\newcommand{\hull}{\operatorname{\mathsf{span}}}	
\newcommand{\Aut}{\operatorname{\mathsf{Aut}}}		
\renewcommand{\dim}{\operatorname{\mathsf{dim}}}	
\DeclarePairedDelimiter\paren{(}{)}					
\DeclarePairedDelimiter\set{\{}{\}}					
\DeclarePairedDelimiterX\bil[2]{\langle}{\rangle}{#1,#2}
\DeclarePairedDelimiterX\range[2]{\{}{\}}{#1,\dots,#2}	
\DeclarePairedDelimiterX\setb[2]\lbrace\rbrace{#1 \;\delimsize\vert\; #2}	
\renewcommand{\subset}{\subseteq}
\renewcommand{\epsilon}{\varepsilon}
\newcommand{\conj}{\overline}			
\newcommand{\xysquare}[8]{
\xymatrix{
#1 \ar[r]^-{#5} \ar[d]_{#6} & #2 \ar[d]^{#7} \\
#3 \ar[r]_-{#8} & #4 }
}
\newcommand{\rest}[2]{\left.{#1}\right\vert_{#2}}	
\hypersetup{
unicode			= true,
breaklinks		= true,
pdfborderstyle	= {/S/D/D [3 2]/W 1},
pdftitle		= {On the local isomorphism property for families of K3 surfaces},
pdfauthor		= {Tim Kirschner (University of Duisburg-Essen)},
pdfkeywords		= {},
pdfprintscaling	= None
}
\author{Tim Kirschner}
\address{Fakultät für Mathematik\\ Universität Duisburg-Essen}
\email{tim.kirschner@uni-due.de}
\urladdr{\url{http://www.esaga.uni-due.de/tim.kirschner/}}
\thanks{This work has been supported by the SFB/Transregio 45 of the DFG and by the Korea Institute for Advanced Study}
\begin{document}
\title{On the local isomorphism property for families of K3 surfaces}
\date\today
\maketitle

\section{Introduction}
\label{intro}

Let $f \colon X \to S$ and $f' \colon X' \to S$ be two families of compact complex manifolds over a reduced complex space $S$. Following Meersseman \cite[496]{Mee11} we say that these families are \emph{pointwise isomorphic} when for all points $s \in S$ the fibers $X_s = f^{-1}(s)$ and $X'_s = f'^{-1}(s)$ of the families are biholomorphic.
When $U$ is an open complex subspace of $S$, we write $f_U \colon X_U \to U$ for the induced holomorphic map where $X_U \subset X$ denotes the inverse image of $U \subset S$ under $f$; we adopt the analogous notation for $f'$.
Given a point $o \in S$ we say that the family $f'$ is \emph{locally isomorphic} at $o$ to $f$ when there exist an open complex subspace $T \subset S$ and a biholomorphism $g \colon X_T \to X'_T$ such that $o \in T$ and $f_T = f'_T \circ g$.
We say that the family $f$ has \emph{the local isomorphism property} at $o$ when all families that are pointwise isomorphic to $f$ are locally isomorphic at $o$ to $f$.

In 1977, assuming $S$ nonsingular, Wehler states \cite[77]{Weh77} that it is unclear whether $f$ has the local isomorphism property at all points of $S$ if the function
\begin{equation}\label{h0}
s \mapsto \hdim0{X_s}{\Theta_{X_s}} \coloneqq \dim_\CC \Hsh0{X_s}{\Theta_{X_s}} = \dim\Aut(X_s)
\end{equation}
is constant on $S$.
Meersseman \cite[Theorem 3]{Mee11} asserts that Wehler's question becomes a valid criterion. We contend that the opposite is true.

\begin{theo}\label{theorem}
There exist two families of K3 surfaces over a complex manifold $S$ together with a point $o \in S$ such that the families are pointwise isomorphic but not locally isomorphic at $o$.
\end{theo}

Recall \cite[\S1.3.3]{K3book} that when for all $s \in S$ the fiber $X_s$ of $f$ is a K3 surface, the assignment of \cref{h0} defines the identically zero function.

\subsection*{Acknowledgements}
I would like to thank Martin Schwald for numerous valuable discussions on the subject. Moreover, I would like to thank Jun-Muk Hwang for his hospitality at the Korea Institute for Advanced Study.

\section{Construction of the families}
\label{construction}

We start with a K3 surface $F$ and a $(-2)$-class $d$ on $F$; that is, $d \in \Hdg11F \subset \Hsh2F\CC$ is an integral cohomology class with $\bil dd = -2$. The angle brackets denote the topological intersection form on $\Hsh2F\CC$, which is given by the cup product and the evaluation at the homology class that determines the orientation of $F$. Explicitly $F$ could be the Fermat quartic in $\CP3$ and $d$ the class of a projective line that is contained in $F$.

Observe that the group $\Hsh2F{\Theta_F}$ is trivial \cite[\S6.2.3]{K3book}.
Hence by the theorem of Kodaira, Nirenberg, and Spencer \cite[452]{KNS58} there exist a family of compact complex manifolds $f \colon X\to S$ over a complex manifold $S$ and a point $o \in S$ such that the Kodaira--Spencer map
\[
\rho_o \colon \Tsp So \to \Hsh1{X_o}{\Theta_{X_o}}
\]
is an isomorphism and $F \cong X_o$.
In fact the quoted theorem yields more---namely, that $S$ is, setting $m \coloneqq \hdim1F{\Theta_F}$, an open ball in $\CC^m$ and that topologically $f$ is nothing but the Cartesian projection $\pr1 \colon S\times F \to S$.
In particular for all $s \in S$ the identity map on $F$ induces a vector space isomorphism
\[
\mu_s \colon \Hsh2{X_s}{\CC} \to \Hsh2{X_o}{\CC} \eqqcolon V
\]
which restricts to an isomorphism between the integral cohomology groups and is an isometry with respect to the intersection forms. Notice that the orientation of the fibers varies locally constantly in a family of compact complex manifolds.
Shrinking the ball $S$ we can assume that $X_s$ is a K3 surface for all $s \in S$ \cite[loc.~cit.]{K3book}.

Let $\PP V$ and $Q \subset \PP V$ denote the projective space of $1$-dimensional subspaces of $V$ and the projective quadric defined by the intersection form on $V$, respectively. Then by the local Torelli theorem for K3 surfaces \cite[Proposition 6.2.8]{K3book}, the assignment $s \mapsto \mu_s[\Hdg20{X_s}]$ yields a holomorphic map $p \colon S \to Q$, the \emph{period map}, which is a local biholomorphism at $o$.
We use square brackets to emphasize that we are taking the image of a subspace by a map rather than the image of a point.
Shrinking $S$ we can assume that $p$ is an open embedding. Thus replacing $f$ and $S$ by $p\circ f$ and $p(S)$, respectively, we can assume that $S \subset Q$ is an open subspace on which the period map is the identity:
\begin{equation}\label{pm-id}
\mu_s[\Hdg20{X_s}] = s, \qquad \forall s\in S.
\end{equation}

By virtue of the canonical identification $F \cong \set o\times F = X_o$ we regard $d$ as an element of $V$ and consider the reflection map
\[
\phi \colon V \to V, \qquad \phi(v) = v + \bil vdd.
\]
The map $\phi$ is an involutory linear isometry of $V$ with respect to the intersection form. Hence we obtain an involutory biholomorphism $h \colon Q \to Q$ satisfying $h(s) = \phi[s]$ for all $s \in Q$.
Since
\[
o = \mu_o[\Hdg20{X_o}] = \Hdg20{X_o} \perp \Hdg11{X_o} \ni d,
\]
we see that $o$ is a fixed point of $h$. Thus $h^{-1}(S) \subset Q$ is an open subspace with $o \in h^{-1}(S)$. Replacing $S$ by $S \cap h^{-1}(S)$, we may assume that the map $h$ restricts to a biholomorphism $i \coloneqq \rest hS \colon S \to S$. We define
\[
f' \coloneqq i \circ f \colon X' \coloneqq X \to S.
\]
It is evident that $f'$ is then, just like $f$, a family of K3 surfaces over the complex manifold $S$.

\section{Verification of the isomorphism properties}

In the first place, we show that $f$ and $f'$ are pointwise isomorphic. For that matter let $t \in S$ be an arbitrary point. The definition of $f'$ implies that $X'_t = X_{i^{-1}(t)}$. Setting
\[
\psi \coloneqq \mu_t^{-1} \circ \phi \circ \mu_{i^{-1}(t)} \colon \Hsh2{X_{i^{-1}(t)}}\CC \to \Hsh2{X_t}\CC
\]
and applying \cref{pm-id}, we see that
\[
\psi[\Hdg20{X_{i^{-1}(t)}}] = (\mu_t^{-1} \circ \phi)[i^{-1}(t)] = \mu_t^{-1}[t] = \Hdg20{X_t}.
\]
In addition, $\phi$ and whence $\psi$ is an isometry with respect to the intersection forms and restricts to an isomorphism between the integral cohomology groups. Therefore $X_t$ and $X'_t$ are biholomorphic by virtue of the (weak form of the) global Torelli theorem for K3 surfaces \cite[Theorem 7.5.3]{K3book}.

In the second place, we show that $f'$ is not locally isomorphic at $o$ to $f$.

\begin{lemm}\label{linalg}
Let $\alpha \colon V \to V$ be a linear automorphism and $U \subset Q$ be a nonempty open subset such that $\alpha[t] = t$ for all $t\in U$. Then $\alpha = \lambda\id V$ for a nonzero complex number $\lambda$.
\end{lemm}
\begin{proof}
We will use that the rank of the intersection form on $V$ is $\ge3$ so that $Q$ is an irreducible quadric in $\PP V$. Indeed, of course, we know that $n \coloneqq \dim_\CC V=22$ and that the intersection form is nondegenerate \cite[\S1.3.3]{K3book}.

Assume that $W\subset V$ is a linear subspace and $U \subset \PP W$. Then, as $Q$ is irreducible and $U\subset Q$ is nonempty and open, the identity theorem for holomorphic functions entails that $Q \subset \PP W$. Since a quadric of rank $\ge2$ is not contained in a linear hyperplane, we conclude that $W=V$. This means that the affine cone over $U$ spans the vector space $V$. 
Hence there exists a basis $(v_1,\dots,v_n)$ of $V$ such that $\CC v_k \in U$ for all $k \in \range1n$. For every $l \in \range1n$ define
\[
W_l \coloneqq \hull\paren*{\setb{v_k}{k\ne l}} \subset V.
\]
Then $U \setminus \PP{W_1}$ is nonempty and open in $Q$. Repeating this argument inductively, we conclude that
$
U \setminus \bigcup_{l=1}^n \PP{W_l}
$
is nonempty and open in $Q$. In particular there exists a vector $v = \mu_1v_1 + \dots + \mu_nv_n$ in $V$ such that $\CC v \in U$ and $\mu_l\ne0$ for all $l$.

Since $\alpha[t]=t$ for all $t \in U$, there exist nonzero complex numbers $\lambda_1,\dots,\lambda_n$ and $\lambda$ such that $\alpha(v_k) = \lambda_kv_k$ for all $k$ and $\alpha(v) = \lambda v$. Therefore
\begin{align*}
(\mu_1\lambda_1)v_1 + \dots + (\mu_n\lambda_n)v_n &= \mu_1\alpha(v_1) + \dots + \mu_n\alpha(v_n) \\
&= \alpha(v) \\
&= \lambda v \\
&= (\lambda\mu_1)v_1 + \dots + (\lambda\mu_n)v_n.
\end{align*}
Comparing coefficients we see that $\mu_k\lambda_k = \lambda\mu_k$, whence $\lambda_k = \lambda$, for all $k$. This proves that $\alpha = \lambda\id V$.
\end{proof}

Now assume $f'$ is locally isomorphic at $o$ to $f$. By \cref{intro} we dispose of an open subspace $T \subset S$ and a biholomorphism $g \colon X_T \to X'_T$ such that $o \in T$ and $f_T = f'_T \circ g$. Fix a point $t \in T$. Then $g$ induces a biholomorphism $g_t \colon X_t \to X'_t$.
Observing that $X'_T = X_{i^{-1}(T)}$, we obtain a commutative diagram of linear maps:
\[
\xysquare{\Hsh2{X_{i^{-1}(T)}}\CC}{\Hsh2{X_T}\CC}{\Hsh2{X_{i^{-1}(t)}}\CC}{\Hsh2{X_t}\CC}
{g^*}{\mathsf{rest.}}{\mathsf{rest.}}{g_t^*}
\]
Without loss of generality we may assume that $T$ is biholomorphic to a ball in $\CC^m$. This makes the vertical arrows in the diagram isomorphisms.
Employing the commutativity of the diagram twice, once for $t$ and once for $t=o$, we deduce that
\[
g_o^* \circ \mu_{i^{-1}(t)} = \mu_t \circ g_t^*.
\]
Invoking \cref{pm-id} we furthermore deduce that
\begin{align*}
(g_o^*\circ\phi^{-1})[t] &= g_o^*[i^{-1}(t)] \\
&= (g_o^* \circ \mu_{i^{-1}(t)})[\Hdg20{X_{i^{-1}(t)}}] \\
&= (\mu_t \circ g_t^*)[\Hdg20{X_{i^{-1}(t)}}] \\
&= \mu_t[\Hdg20{X_t}] = t.
\end{align*}
Hence \cref{linalg} implies that $g_o^* \circ \phi^{-1} = \lambda\id V$ for a complex number $\lambda\ne0$. Seeing that $g_o^* \circ \phi^{-1}$ restricts to an automorphism of $\Hsh2{X_o}\ZZ$, we infer that $\lambda$ is either $1$ or $-1$. Both alternatives lead to a contradiction.

Consider the real cone
\[
C \coloneqq \setb{v \in \Hdg11{X_o}}{\conj v = v,\; \bil vv>0} \subset \Hdg11{X_o}
\]
and notice that both $g_o^*$ and $\phi$ map $C$ homeomorphically onto itself. Notice moreover that $C$ has precisely two connected components; one of these components, say $C_+$, contains all Kähler classes of $X_o$ while the other component is just $-C_+$ \cite[\S8.5.1]{K3book}.
Since $g_o^*$ takes Kähler classes to Kähler classes, $g_o^*$ preserves the connected components of $C$. For all $w \in C$ we see that the line segment joining $w$ and $\phi(w)$ is contained in $C$.
Thus $\phi$ preserves the connected components of $C$, too, while $-\phi$ swaps the components. So if $g_o^* = -\phi$, we obtain a contradiction.
If on the other hand $g_o^* = \phi$, then $g_o^*(d) = -d$. As mentioned at the outset of \cref{construction} we can assume that $d$ is the class of a smooth rational curve on $X_o$. Then, however, for every Kähler class $c$ on $X_o$,
\[
0 < \bil cd = \bil{g_o^*(c)}{g_o^*(d)} = -\bil{g_o^*(c)}d < 0,
\]
which completes our proof of \cref{theorem}.

\section{Closing remarks}

In our discussions with Schwald we have realized that an alternative, yet effectively related, proof strategy for \cref{theorem} uses the existence of the flop of a $(-2)$-curve in a complex threefold. Using this idea we can construct pointwise but not everywhere locally isomorphic families of K3 surfaces over a $1$-dimensional complex manifold $S$, whereas the complex manifold $S$ in \cref{construction} is of dimension $20$. We refrain from going into the details.

As regards Meersseman's work it seems worthwhile to investigate whether and to what extend a weakened form of his result \cite[Theorem 3]{Mee11} remains true—for instance, assuming that \cref{h0} defines a constant map, does the family of compact complex manifolds $f$ have the local isomorphism property at the (very) general point of $S$?

\printbibliography
\end{document}